\documentclass[11pt,a4paper]{amsart}
\usepackage{graphicx} 
\usepackage{amsmath, amsfonts, amssymb, amsthm, comment}
\usepackage{color}
\title[Totally Real Parameters in Generalized Mandelbrot Set]{Totally real algebraic numbers in  generalized Mandelbrot set}

\author{Kevin G. Hare}
\address{University of Waterloo \\
Department of Pure Mathematics \\
Waterloo, Ontario \\
Canada  N2L 3G1}
\email{kghare@uwaterloo.ca}
\thanks{Research of K.G. Hare was supported, in part, by NSERC Grant 2019-03930}
\author[Chatchai Noytaptim]{Chatchai Noytaptim}
\address{University of Waterloo \\
Department of Pure Mathematics \\
Waterloo, Ontario \\
Canada  N2L 3G1}
\email{cnoytaptim@uwaterloo.ca, chatchai.noytaptim@gmail.com}
\date{\today}
\newtheorem{thm}{Theorem}[section]
\newtheorem{prop}{Proposition}[section]
\newtheorem{lemma}{Lemma}[section]
\newtheorem{rem}{Remark}[section]

\begin{document}

\begin{abstract}
 In this article, we study some potential theoretical and topological aspects of the generalized Mandelbrot set introduced by Baker and DeMarco. 
 For $\alpha$ real, we study the set of all totally real algebraic parameters $c$ such that $\alpha$ is preperiodic under the iteration of the one-parameter family $f_c(x) = x^2 + c$.
 We show that when $|\alpha| < 2$ and rational then the set of totally real algebraic parameters $c$ with this property is finite, whereas if $|\alpha| \geq 2$ and rational then this set is countably infinite.  
 As an unexpected consequence of this study, we also show that when $|\alpha| \geq 2$ then parameters $c$ such that $\alpha$ is $f_c$-periodic are necessarily real. 
 As a special case, we classify all totally real algebraic integers $c$ such that $\alpha = \pm1$ is preperiodic.
\end{abstract}

\maketitle

\section{Introduction} 
 In this article, we study  dynamics of the quadratic one-parameter family of the form $f_c(x)=x^2+c\in\overline{\mathbb{Q}}[x]$. Recall that $\alpha\in\overline{\mathbb{Q}}$ is a \textbf{preperiodic point} for $f_c$ if the set $\{\alpha, f_c(\alpha), f^2_c(\alpha),...\}$ of forward iteration of $\alpha$ contains finite distinct values. In other words, $$f^n_c(\alpha)=f_c^m(\alpha)$$ for some integers $0\leq m<n$. Here we write $f_c^n$ to represent
$f_c^n:=f_c\circ f_c \circ...\circ f_c$ the $n$-fold composition of $f_c$ with itself.

In \cite{BD11}, Baker and DeMarco defined the generalized (complex) Mandelbrot set $\mathcal{M}(\alpha)$ as follows
$$\mathcal{M}(\alpha):=\{c\in\mathbb{C}: |f^n_c(\alpha)|\nrightarrow\infty\,\,\text{as}\,\,n\rightarrow\infty\}.$$ That is, it is the collection of complex parameters $c\in\mathbb{C}$ whose forward orbit of $\alpha$ remains bounded under the iteration. This notion generalizes the classical Mandelbrot set $\mathcal{M}(0)$ and plays an important role in the study of unlikely intersection problems in complex dynamics (see e.g., \cite{BD11} and \cite{GHT13}). 

Our first three results, given in Section \ref{sec:genmandel}, gives an insight of $\mathcal{M}(\alpha)$ topologically.  

\begin{thm} \label{thm:M is bounded}
Let $R_{\alpha} = |\alpha|^2+1+\sqrt{|\alpha|^2+1}$.
Then 
\begin{enumerate}
\item If $\alpha \in \mathbb{C}$ then $\mathcal{M}(\alpha)\subseteq \overline{D(0,R_{\alpha})}$. 
\item If $\alpha \in \mathbb{R}$ then 
$$\mathcal{M}(\alpha)\cap\mathbb{R}\subseteq \begin{cases}[-R_{\alpha},1/4],&\mbox{when}\,\,\,|\alpha|\leq 1/2\\ [-R_{\alpha},|\alpha|-\alpha^2],&\mbox{when}\,\,\,|\alpha|>1/2\end{cases}.$$
\end{enumerate}
\end{thm}

\begin{rem}
    The equality occurs only when $\alpha=0$.  More precisely, the real part $\mathcal{M}(0)\cap\mathbb{R}=[-2,1/4]$ (cf. \cite[Ch.VIII, Theorem 1.2]{CG93}).
\end{rem}

\begin{thm} \label{thm:M is real}
Let $\alpha \in \mathbb{R}$ with $\alpha \in (-\infty, -2] \cup [2, \infty)$.
Then $\mathcal{M}(\alpha)$ is totally disconnected and is contained in the real line.
\end{thm}
\begin{thm} \label{thm:M no isolated points}
Let $\alpha \in \mathbb{C}$.  Then $\mathcal{M}(\alpha)$ does not have any isolated points.
\end{thm}

It is well-understood by complex dynamicists that there are infinitely many complex parameters $c\in\mathbb{C}$ for which $\alpha\in\mathbb{C}$ is $f_c$-preperiodic (cf. \cite[Lemma 3.5]{BD11}). In particular, fix $\alpha\in\overline{\mathbb{Q}}$, the collection of all parameters $c\in\overline{\mathbb{Q}}$ such that $\alpha$ is a preperiodic point of $f_c$ forms a set of bounded absolute (Weil) height.  In fact, any complex parameter $c\in\mathbb{C}$ satisfying the polynomial  \begin{equation}
    F_{m,n}(c)=f^n_c(\alpha)-f^m_c(\alpha)\in \overline{\mathbb{Q}}[c]\label{eq:prepparameter}\end{equation} for some integers $0\leq m<n$, must indeed be an algebraic number. On the other hand, all roots ($c$ and all of its conjugates) of  $F_{m,n}(c)$ are contained in the  generalized Mandelbrot set $\mathcal{M}(\alpha)$. As each parameter $c$ in $\mathcal{M}(\alpha)$ has  absolute value at most $R_{\alpha}$, we have that the absolute height of algebraic parameter $c$ satisfying equation (\ref{eq:prepparameter}) is bounded.  
By Northcott's theorem, it immediately yields for any fixed number field $K$ that there are only 
    finitely many algebraic parameters $c \in K$ such that $\alpha\in\overline{\mathbb{Q}}$ is $f_c$-preperiodic.
    
Recall that an algebraic number $\theta\in\overline{\mathbb{Q}}$ is said to be \textbf{totally real} if its minimal polynomial has all distinct real roots. 
Denote by $\mathbb{Q}^{\mathrm{tr}}$ the collection of all totally real algebraic numbers. 
Let $\alpha\in\mathbb{C}$. Denote by
$$\mathrm{Prep}(\alpha)=\{c \in\mathbb{C} :  \alpha\,\,\text{is preperiodic for} \,\,f_c(x)=x^2+c\}$$ the collection of parameters $c$ satisfying $f^n_c(\alpha)=f^m_c(\alpha)$ for some integers $0\leq m<n$. 
One goal of this article is to consider the size of $\mathrm{Prep}(\alpha)\cap \mathbb{Q}^{\mathrm{tr}}$ for rational $\alpha$.
We prove

\begin{thm} \label{thm:alpha}
Let $\alpha \in \mathbb{Q}$.
\begin{enumerate}
\item If $|\alpha| < 2 $ then there are finitely many totally real algebraic parameters $c$ such that $\alpha$ is $f_c$-preperiodic.
    Further more this set is always non-empty.
    \label{case:alpha<2}
    \item If $|\alpha| \geq 2$ then there are infinitely many totally real algebraic parameters $c$ such that $\alpha$ is $f_c$-preperiodic.
    \label{case:alpha>=2}
\end{enumerate}
\end{thm}

It is worth pointing out that the finiteness result in Theorem \ref{thm:alpha} part \eqref{case:alpha<2} does not follow directly from the Northcott's theorem as explained above. Indeed, each totally real algebraic parameters $c$ belongs to the maximal totally real $\mathbb{Q}^{\mathrm{tr}}$ extension of $\mathbb{Q}$  which is an infinite field.

The major ingredient and motivation in the proof of Theorems \ref{thm:alpha} above and \ref{thm:alpha = 1} below is an adelic Fekete-Szeg\"{o}'s theorem (Theorem \ref{thm:FeketeSzego}). In order to apply Theorem \ref{thm:FeketeSzego}, we need to understand the geometric shape of the complex generalized Mandelbrot set (see Theorem \ref{thm:M is bounded}) introduced by Baker-DeMarco. Indeed, all totally real algebraic parameters $c$ (with all of its conjugates) for which $\alpha$ is $f_c$-preperiodic are contained in the real part of the generalized Mandelbrot set (Lemma \ref{lem:prepinmandel}). Roughly speaking, for  rational numbers $\alpha\in (-2,2)$, such real part is small in the sense of arithmetic capacity theory and the Fekete-Szeg\"{o}'s theorem yields the desired conclusion in part (1). On the other hand, the conclusion of Theorem \ref{thm:alpha} part \eqref{case:alpha>=2} is a consequence of the fact that the generalized Mandelbrot set is contained in real line when $\alpha\in\mathbb{Q}\backslash (-2,2)$ (Theorem \ref{thm:M is real}).

As a special case we also show

\begin{thm} \label{thm:alpha = 1}
    Let $\alpha = \pm 1$.  Then $-2-\sqrt{2},-3,-2,-1,-2+\sqrt{2}$ are the only totally real algebraic parameters $c$ such that $\alpha$ is $f_c$-preperiodic.
\end{thm}

We are able to explicitly compute all totally real algebraic integers $c$ so that the set $\{1, f_c(1), f^2_c(1),...\}$ and the set $\{-1, f_c(-1), f^2_c(-1),...\}$ contain finitely many distinct values.  
The computation heavily hinges on a numerical tool (Theorem \ref{thm : numerical tool}) to help bound the degree of any algebraic integer with all of conjugates lying in an interval of length less than $4$. This is done in Theorem \ref{thm:alpha = 1}.

This article is organized as follows. In Section \ref{sec:Background}, we  introduce important terminology and basic results from complex and $p$-adic potential theory and review adelic version of generalized Mandelbrot set, an adelic Fekete-Szeg\"{o} theorem, and a numerical criterion. In Section \ref{sec:genmandel}, we investigate related topological properties of the generalized Mandelbrot set, proving Theorems \ref{thm:M is bounded}, \ref{thm:M is real} and \ref{thm:M no isolated points}.  
In Section \ref{sec:Main} we investigate totally real algebraic numbers $c$ such that $\alpha$ is perperiodic.  In particular, we proof Theorems \ref{thm:alpha} and \ref{thm:alpha = 1}.
The last section, Section \ref{sec:conc}, gives some open questions and arose as a result of this research.

\section{Background}
\label{sec:Background}

\subsection{Potential Theory} 
Let $K$ be a number field. Then $K_v$ is the completion of $K$ with respect to an absolute value $|\cdot|_v$
 and $\mathbb{C}_v$ is the completion of a fixed algebraic closure $\overline{K}_v$ of $K_v$. Note that $\mathbb{C}_v$ is both complete and algebraically closed. We fix an extension of $|\cdot|_v$ on $K_v$ to an absolute value on $\mathbb{C}_v$ and we still denote by $|\cdot|_v$ by abuse notation. At the Archimedean place $v=\infty$, $\mathbb{C}_v$ is homeomorphic to the usual complex plane $\mathbb{C}$. For more details about non-Archimedean dynamics, see \cite{Be19} and \cite{Ru02}.

For each compact subset $E_v$ of $\mathbb{C}_v$, the (local) logarithmic capacity of $E_v$, relative to $\infty$,  (denoted  $\gamma_{\infty}(E_v)$) is defined by
\begin{equation*}-\log_v \gamma_{\infty}(E_v)=\inf_{\mu}\iint_{E_v\times E_v}-\log_v|x-y|_v\mu(x)\mu(y)\label{eq:capacityandenery}\end{equation*} where the infimum is taken over all  probability measures $\mu$ supported on $E_v$ and $\log_v(\cdot)$ denotes logarithmic base $q_v$ (here $q_v$ is the order of the residue field of $K_v$). Thus for the Archimedean $v$, $\log_v(x)=\log(x)=\ln(x).$
 Recall that a compact adelic set is the set of the form
 $$\mathbb{E}=\prod_{v\in M_K}E_v$$ where $E_v$ is a nonempty compact subset of $\mathbb{C}_v$ for each place $v$ of $K$ and $E_v=\mathcal{D}(0,1)=\{x\in\mathbb{C}_v : |x|_v\leq 1\}$--closed unit disc in $\mathbb{C}_v$--for all but finitely many places $v$. Thus the global logarithmic capacity of $\mathbb{E}$ (relative to $\infty$) is given by
 $$\gamma_{\infty}(\mathbb{E})=\prod_{v\in M_K}\gamma_{\infty}(E_v)^{N_v}$$ where $N_v=[K_v:\mathbb{Q}_v]\geq1$ (they are the same normalization of absolute values $|\cdot|_v$ appears on the product formula of $K$). Thus the product on the right-side must be finite as except for finitely many terms $\gamma_{\infty}(E_v)=\gamma_{\infty}(\mathcal{D}(0,1))=1$ (cf. \cite[Example 5.2.16]{Ru89}).
 
Let $E\subset\mathbb{C}$ be compact. In the Archimedean setting, the logarithmic capacity of $E$ can alternatively be defined as the limiting of the $n$-diameter of $E$.  The \textbf{$n$-diameter} of $E$ is given by
$$d_n(E)=\sup_{x_1,...,x_n\in E}\prod_{i<j}|x_i-x_j|^{2/n(n-1)}.$$ 
It was shown by Fekete-Szeg\"{o} that the sequence $d_n(E)$ is decreasing and the limit exists. It is known as the logarithmic capacity $\gamma_{\infty}(E)$ of $E$ (with respect to $\infty$) , see \cite[Theorem 5.5.2]{Ra95}. That is,
$$\lim_{n\rightarrow\infty}d_n(E)=\gamma_{\infty}(E).$$  The capacity of some compact subsets of $\mathbb{C}$ are known. For instance, Ransford \cite[Corollary 5.4]{Ra95} shows that the logarithmic capacity of a closed real interval is one-quarter of its length. 

\subsection{Adelic Generalized Mandelbrot Set} 
Baker-DeMarco also adelically generalized the definition of the Mandelbrot set, for each place $v$ of $K$,  by
$$\mathcal{M}_v(\alpha)=\{c\in\mathbb{C}_v : \sup_{n\in\mathbb{N}}|f^n_c(\alpha)|_v<+\infty\}.$$ When $v=\infty$ and $\alpha=0$, it coincides with the usual Mandelbrot set. To ease notation, we write $\mathcal{M}(\alpha)$ in place of $\mathcal{M}_{\infty}(\alpha)$ for Archimedean place $v=\infty$. For each non-Achimedean place $v\in M_K$, it is also useful to point out that $\mathcal{M}_v(\alpha)=\mathcal{D}(0,1)$ when $|\alpha|_v\leq 1$ by the strong triangle inequality.

One outstanding potential-theoretical property of the set $\mathcal{M}_v(\alpha)$, for each $v\in M_K$, is that it has logarithmic capacity $1$ (cf. \cite[Proposition 3.3(4)]{BD11} and \cite[Proposition 3.7(2)]{BD11}). To make this note more-or-less self-contained, we provide a more elementary proof for the Archimedean generalized Mandelbrot set $\mathcal{M}(\alpha)$ without relying on B\"{o}ttcher coordinates and Green's function associated to $\mathcal{M}(\alpha)$. 
The following Archimedean computation is similar to that of \cite[\S8]{BH05}. 

\begin{prop}\label{prop: capmandel} For each place $v$ of $K$, 
    the $v$-adic logarithmic capacity of the generalized Mandelbrot set $\mathcal{M}_v(\alpha)$ is $1$. That is, $\gamma_{\infty}(\mathcal{M}_v(\alpha))=1.$
\end{prop}

\begin{proof} For each $n\in\mathbb{N}$, we define $\mathcal{M}_n(\alpha):=\{c\in\mathbb{C} : |f^n_c(\alpha)|\leq R_{\alpha}\}$ where $R_{\alpha}$ is the radius of the smallest disk centered at $0$ that contains $\mathcal{M}(\alpha)$ (cf. Theorem \ref{thm:M is bounded}). Then we see that 
$$\mathcal{M}_1(\alpha)\supseteq\mathcal{M}_2(\alpha)\supseteq...\quad\text{and}\quad \mathcal{M}(\alpha)=\bigcap_{n\geq 0}\mathcal{M}_n(\alpha).$$ We write $\psi_n(X):=f_X^n(\alpha)$. Then $\psi_n(X)$ is a monic polynomial of degree $2^{n-1}$ in $X$. Thus $\mathcal{M}_n(\alpha)=\{x\in\mathbb{C} : |\psi_n(x)|\leq R_{\alpha}\}.$ The logarithmic capacity of each polynomial lemniscate  $\mathcal{M}_n(\alpha)$ is $R^{1/2^{n-1}}_{\alpha}$ (cf. \cite[Theorem 5.2.5]{Ra95}). Since the logarithmic capacity is continuous with respect to a monotone decreasing sequence (cf. \cite[Theorem 5.1.3(a)]{Ra95}), so we have that
$$\gamma_{\infty}(\mathcal{M}(\alpha))=\lim_{n\rightarrow\infty}\gamma_{\infty}(\mathcal{M}_n(\alpha))=\lim_{n\rightarrow\infty}R^{1/2^{n-1}}_{\alpha}=1.$$
For the non-archimedean computation, we refer the reader to \cite[Proposition 3.7(2)]{BD11}.
\end{proof}

\subsection{Fekete-Szeg\"{o} Theorem} 
The Fekete-Szeg$\ddot{\text{o}}$'s theorem plays a major role in the proof of Theorem \ref{thm:alpha}. The reader may consult to the standard textbooks, e.g. \cite[Theorem 6.2.8]{BR10}, \cite[Theorem 6.3.1, Theorem 6.3.2]{Ru89}, and  \cite[Theorem 0.4]{Ru13} for detailed exposition.
\begin{thm}\label{thm:FeketeSzego} (Fekete-Szeg$\ddot{\text{o}}$) Let $K$ be a number field and let $\mathbb{E}=\prod_{v\in M_K} E_v$ be an (affine) adelic set such that each $E_v$ stable under continuous automorphisms of $\mathbb{C}_v/K_v$. Then
\begin{enumerate}
\item If $\gamma_{\infty}(\mathbb{E})<1$, then there exists an adelic neighborhood $\mathbb{U}$ of $\mathbb{E}$ which contains finitely many complete conjugate sets of points in $\overline{K}$.
\item If $\gamma_{\infty}(\mathbb{E})\geq 1$, then every  adelic neighborhood $\mathbb{U}$ of $\mathbb{E}$ contains infinitely many complete conjugate sets of points in $\overline{K}$.
\end{enumerate}
\end{thm}

Here an adelic neighborhood  $\mathbb{U}$ of $\mathbb{E}$ is the set of the form $\mathbb{U}=\prod_{v\in M_K} U_v$ such that $U_v$ is a set in $\mathbb{C}_v$ containing $E_v$.

\begin{rem} The following useful observations---regarding results of Fekete-Szeg$\ddot{\text{o}}$---were pointed out by Rumely \cite[$\S 0$]{Ru89}. Suppose $E\subset\mathbb{C}$ is compact.
    \begin{enumerate}
    \item The neighborhood appears in the statement of Fekete-Szeg$\ddot{\text{o}}$'s theorem because the compact $E$ itself may not contain algebraic integers at all. For example, consider the outer boundary of the closed disk $\overline{D(0,\pi)}$ which is the circle $E=\{x\in\mathbb{C} : |x|=\pi\}$ with transcendental radius $\pi>1$. The logarithmic capacity $\gamma_{\infty}(E)=\gamma_{\infty}(\overline{D(0,\pi)})$ is $\pi$ (cf. \cite[Corollary 5.2.2]{Ra95}). Thus it is necessary to pass to a neighborhood of the compact set $E$.
    \item In the case that compact set $E$ is the closure of its interior, we may replace a neighborhood of $E$ by $E$ itself.
    \end{enumerate}
\end{rem}

\subsection{Numerical Criterion} 
In proving Theorem \ref{thm:alpha} part \eqref{case:alpha<2}, we rely on the numerical tool Theorem \ref{thm : numerical tool}. It basically says that in any interval of length less than $4$ (which equivalent to saying that the interval has logarithmic capacity less than $1$), we can bound the degree of any algebraic integer whose conjugates completely contained in the interval. It is a useful tool in searching algebraic integers of small degree. 

\begin{thm} \label{thm : numerical tool}  \cite[Theorem 8]{NP23} Let $[\alpha,\beta]$ be a real interval of length less than $4$. Define sequences $\{a_n\}_{n\geq 2}$ and $\{b_n\}_{n\geq 2}$ by  $$a_n=d_n([\alpha,\beta])^{n(n-1)}=(\beta-\alpha)^{n(n-1)}D_n\quad\text{and}\quad b_n=\frac{n^{2n}}{n!^2}$$ where the sequence $\{D_n\}_{n\geq 2}$ is defined recursively by $D_2=1$ and  $$D_n=\frac{n^n(n-2)^{n-2}}{2^{2n-2}(2n-3)^{2n-3}}D_{n-1}\quad (n\geq 3).$$ Suppose that there exists an integer $n_0\geq 2$ with the properties that $$a_{n_0}<b_{n_0}\quad \text{and}\quad \frac{a_{n_0+1}}{a_{n_0}}<\frac{b_{n_0+1}}{b_{n_0}}.$$ If $\theta$ is an algebraic integer whose minimal polynomial $F_{\theta}(X)\in\mathbb{Z}[X]$ over $\mathbb{Q}$ has all distinct roots in the interval $[\alpha,\beta]$, then $[\mathbb{Q}(\theta):\mathbb{Q}]<n_0.$
\end{thm}

\begin{rem} We make comments about the statement of the numerical criterion in order. \begin{enumerate}
\item The sequence $\{a_n\}$ describes the supremum of discriminant (cf. the definition of the $n$-diameter) of all monic polynomials of degree $n$ whose roots are contained in the interval $[\alpha,\beta]$ and the sequence $\{b_n\}$ describes the Minkowski's constant of a totally real number field generated by an algebraic number $\theta$ whose minimal polynomial has degree $n$. 
\item The condition on the length of an interval is necessary and cannot be removed which is a consequence of Fekete-Szeg\"{o} theorem. The length of a closed  interval is less than $4$ is equivalent to the logarithmic capacity of the interval is less than $1$. This confirms the finiteness of algebraic integers with all of conjugates in such  interval. Notice that the number $4$ is optimal in the sense that the close real interval  with rational end points $[-2,2]$ contains infinitely many algebraic integers of the form $\zeta+\zeta^{-1}$ where $\zeta$ is a root of unity (cf. \cite{Kro57}).
\end{enumerate}
\end{rem}

\section{Baker-DeMarco's generalized Mandelbrot set}\label{sec:genmandel}

In this section we prove Theorems \ref{thm:M is bounded}, \ref{thm:M is real} and \ref{thm:M no isolated points}.
Let $\alpha$ be a complex number. Recall that  Baker and DeMarco \cite[\S3]{BD11}  introduced a (complex) generalized Mandelbrot set as follows
$$\mathcal{M}(\alpha)=\left\{c\in\mathbb{C}:  \sup_{n\in\mathbb{N}}|f^n_c(\alpha)|<+\infty\right\}.$$  It is the collection of all complex parameters $c$ whose forward orbit of $\alpha$ remains bounded with respect to the usual complex absolute value under iteration.  When $\alpha=0$, the set $\mathcal{M}(0)$ is the classical Mandelbrot set. Baker and DeMarco also proved many interesting topological properties of the set $\mathcal{M}(\alpha)$. For example, in \cite[Proposition 3.3]{BD11}, they showed that $\mathcal{M}(\alpha)$ is a compact and full (meaning that it has a complement connected component) subset of $\mathbb{C}$.  Interestingly, it was proved by Favre and Gauthier \cite[Theorem 7.8]{FG22} that $\mathcal{M}(\alpha)$ is connected if and only if $\alpha=0.$  The connectedness locus $\mathcal{M}(0)$ is well-known in the literature whose proof can be found in  \cite{Bea91}, \cite{CG93} and \cite{DH84}. We further investigate some topological properties of the set $\mathcal{M}(\alpha).$ 
\begin{proof}[Proof of Theorem \ref{thm:M is bounded}]  We first verify that the generalized Mandelbrot set $\mathcal{M}(\alpha)$ is contained in a closed of radius $|\alpha|^2+1+\sqrt{|\alpha|^2+1}$ centered at $0$. To see this, we suppose that $|c|>|\alpha|^2+1+\sqrt{|\alpha|^2+1}$. Then there exists $\epsilon>0$ such that $|c|=|\alpha|^2+1+\sqrt{|\alpha|^2+1}+\epsilon.$ Then, by the reverse triangle inequality, we see that \begin{align*}|f_c(\alpha)|&\geq |\alpha|^2-|c|=-\epsilon-1+\sqrt{|\alpha|^2+1}\\ |f^2_c(\alpha)|&\geq |f_c(\alpha)|^2-|c|\geq 2\epsilon+1+\sqrt{|\alpha|^2+1}.\end{align*} Inductively, we have for each $n\geq2$ that
$$|f^n_c(\alpha)|\geq 2(n-1)\epsilon+1+\sqrt{|\alpha|^2+1}.$$ As $n\rightarrow+\infty$, $|f^n_c(\alpha)|\rightarrow+\infty$. Thus the claimed is established.

Next, we compute the real part of $\mathcal{M}(\alpha)\cap\mathbb{R}$ for all real numbers $\alpha$.

\textbf{Case I.} Suppose that $|\alpha|\leq 1/2.$ Using the fact that $f_c$ is an increasing function together with $0\leq |\alpha|\leq1/2$, we have $f_c(0)\leq f_c(|\alpha|)$. Iterating $n$ times, it follows that  \begin{equation}f^n_c(0)\leq f^n_c(|\alpha|).\label{eq:onequarterinduction}\end{equation} If $c>1/4$, then $f^n_c(0)\rightarrow\infty$ as $n\rightarrow\infty$. Hence the inequality (\ref{eq:onequarterinduction}) yields $f^n_c(|\alpha|)\rightarrow\infty$ as $n\rightarrow\infty$. If instead $c=1/4$,we have  $$f^n_c(0)\leq f^n_c(|\alpha|)=\alpha^2+c\leq 1/2$$ for each $n\geq 1$. It is straightforward to check that $f^2_{-R_{\alpha}}(|\alpha|)=f^3_{-R_{\alpha}}(|\alpha|)$. Therefore, $\mathcal{M}(\alpha)\cap\mathbb{R}\subseteq[-R_{\alpha},1/4]$.

\textbf{Case II.} Suppose that $|\alpha|> 1/2.$ Let $C_{\alpha}=|\alpha|-\alpha^2$. If $c>C_{\alpha}$, then we get $f_c(|\alpha|)=\alpha^2+c>\alpha^2+C_{\alpha}=|\alpha|.$ Iterating $f_c$, it gives $f^n_c(|\alpha|)>|\alpha|$. It is clear that $f_c(x)$ is increasing, and so is $f^n_c(x)$ for each $n\geq 2$. This means 
$$|\alpha|<f_c(|\alpha|)<f^2_c(|\alpha|)<...$$ The sequence $\{f^n_c(|\alpha|)\}$ must either go to infinity or converge to one of  fixed points $\beta^{\pm}_{c}=(1\pm\sqrt{1-4c})/2$ of $f_c(x).$ The latter is impossible because $\beta^{\pm}_c<|\alpha|$ for $c>C_{\alpha}.$ The fact that $\mathcal{M}(\alpha)$ is contained in $\overline{D(0,R_{\alpha})}$, we see that $f^n_c(\alpha)$ increases without bound when $c<-R_{\alpha}$ as $n$ becomes arbitrary large. Hence $\mathcal{M}(\alpha)\cap\mathbb{R}\subseteq [-R_{\alpha}, |\alpha|-\alpha^2].$
\end{proof}

\begin{proof}[Proof of Theorem \ref{thm:M is real}]
We note that $f_c(\alpha)$ is even, hence $f_c(\alpha) = f_c(-\alpha)$.
Hence, without loss of generality, we may assume $\alpha > 2$.

To prove this result, we will construct a totally disconnected real Cantor like set $C$ which contains $\mathrm{Prep}(\alpha)$.
As the boundary of $\mathcal{M}(\alpha)$ is contained in the closure $\mathrm{Prep}(\alpha)$, this will prove the result.

We will do this by iteratively constructing nested closed sets $C_0 \supset C_1 \supset C_2 \dots$ on the real line.
We let $C = \cap C_n$.
We will show that for $n > m $ that all roots $c$ of $f^{n}_c(\alpha) - f^{m}_c(\alpha)$ are contained in $C_n \subset \mathbb{R}$.
As $f^{n}_c(\alpha) - f^{m}_c(\alpha)$ is a factor of $f^{n+k}_c(\alpha) - f^{m+k}_c(\alpha)$ for all $k$, this gives that $\mathrm{Prep}(\alpha)$ are real and contained in $C$.

Define $u(c) = \frac{1+\sqrt{1-4c}}{2}$. 
We define \[ C_n = \{c : -u(c) \leq f^{n}_c(\alpha) \leq u(c) \}. \]
We claim that $C_n$ has the desired properties.
That is, we claim $C_{n+1} \subset C_n$ for all $n$, 
    $C \neq \emptyset$, and for all $n > m$ we have the roots of
    $f^{n}_c(\alpha) - f^{m}_c(\alpha)$ are contained in $C_n$.

Consider $C_0$ and $C_1$.
We see that 
\begin{align*}
    C_0 &= \{c: -u(c) \leq f^{0}_c(\alpha) \leq u(c) \} \\
     &= \{c: -u(c) \leq \alpha \leq u(c) \} \\
     & = \{c : c \leq \alpha - \alpha^2\}.
\end{align*}
and
\begin{align*}
    C_1 &= \{c: -u(c) \leq f^{1}_c(\alpha) \leq u(c) \} \\
     &= \{c: -u(c) \leq \alpha^2 + c \leq u(c) \} \\
     & = \{c : -1-\alpha^2 - \sqrt{1+\alpha^2} \leq c \leq \alpha - \alpha^2\}.
\end{align*}
It is worth observing that $c = -1-\alpha^2 - \sqrt{1+\alpha^2}$ is 
    equivalent to $v(c) = f^{0}(\alpha)$.
In general, one can show by induction that 
\begin{align*}
    C_{n+1} & =  \{c: -u(c) \leq f^{n+1}_c(\alpha) \leq u(c) \} \\ 
    & =  \{c: -u(c) \leq f^{n}_c(\alpha) \leq -v(c) \text{ or }
               v(c) \leq f^{n}_c(\alpha) \leq u(c) \}. \\ 
\end{align*}
This proves that $C_{n+1} \subset C_n$.

In general, for $n \geq 1$, $C_n$ will be composed of $2^{n-1}$ 
    intervals.
Consider one of these intervals, $[a_n, b_n]$ contained in $C_n$.
One can check that either we have $f_{a_n}^{n}(\alpha) = u(a_n)$ and $f_{b_n}^{n}(\alpha)= - u(b_{n})$ or we have
$f_{a_n}^{n}(\alpha) = -u(a_n)$ and $f_{b_n}^{n}(\alpha) = u(b_{n})$.
We will assume the first case, the other case is similar.
We note for $c < \alpha - \alpha^2$ that $-u(c) < -v(c) < v(c) < u(c)$.
Hence by the intermediate value theorem there exists $a_n < a_{n+1} < b_{n+1} < b_n$ such that
$f_{a_n}^{n+1}(\alpha) = u(a_n)$,
$f_{a_{n+1}}^{n+1}(\alpha) = -u(a_{n+1})$,
$f_{b_{n+1}}^{n+1}(\alpha) = -u(b_{n+1})$ and 
$f_{b_n}^{n+1}(\alpha) = u(b_n)$.

Let $m \leq n$.
Consider an interval in $C_n$, say $I$.
We see from the above results that $-u(c) \leq f_c^{m}(\alpha), f_c^{n}(\alpha) < u(c)$. 
Further, we see that $f_c^{n}(\alpha)$ take the full possible range.
Hence, by the intermediate value theorem, we have
    $f_c^{n}(\alpha) - f_c^{m}(\alpha)$ contains a root in $I$.

We note that $f_c^{n}(\alpha) - f_c^{m}(\alpha)$ is degree $2^{n-1}$ in $c$.
Further, $C_n$ has $2^{n-1}$ intervals, each of which contains a root
    of $f_c^{n}(\alpha) - f_c^{m}(\alpha)$.
This proves that all roots of $f_c^{n}(\alpha) - f_c^{m}(\alpha)$ are real and contained in $C_n$.
By the previous comments, this implies $\mathrm{Prep}(\alpha)$ 
    is in $C$, and hence $\partial \mathcal{M}(\alpha) \subset C \subset \mathbb{R}$.
As the boundary of $\mathcal{M}(\alpha)$ is real, we see that
    $\mathcal{M}(\alpha)$ has no interior, and is equal to 
    it's boundary, which proves the result.
\end{proof}

\begin{proof}[Proof of Theorem \ref{thm:M no isolated points}]
Assume, for the sake of contradiction, that $\mathcal{M}(\alpha)$ has an isolated point, say $c_0$. 
Thus there exists $\epsilon>0$ such that $D(c_0,\epsilon)\cap\mathcal{M}(\alpha)=\{c_0\}.$ 
Here $D(c_0,\epsilon)=\{c\in\mathbb{C} : |c-c_0|<\epsilon\}$. Let $C = \partial(D(c_0, \epsilon)) =  \{c\in\mathbb{C} : |c-c_0| = \epsilon\}$ be the boundary of $D(c_0, \epsilon)$. 
Then for each $c\in C$ we see that $f_c^n(\alpha)\rightarrow \infty$ as $n\rightarrow\infty$. 
It was shown in Theorem \ref{thm:M is bounded} that $\mathcal{M}(\alpha)$ is bounded and contained in $\mathrm{clos}(D(0,R_{\alpha}))$. 
Thus if $f^n_c(\alpha)>R_{\alpha}$, then $f^n_c(\alpha)\rightarrow\infty$ as $n\rightarrow\infty$. 
For each $n\in\mathbb{N}$, we define $C_n:=\{c\in C : |f^n_{c_0}(c)|>R_{\alpha}\}$. 
It is easy to check that each $C_n$ is a relative open set with respect to $C$. 
Thus $\bigcup_{n\geq 1}C_n=C$. As $C$ is compact, it follows that $C$ has a finite sub-cover. 
There exists $N\in \mathbb{N}$ such that $C_N=C$. 

Let $\phi(c) = f^{N}_c(\alpha)$.
This function is analytic (and in fact a polynomial).  
Hence by the complex open mapping theorem, it maps open sets to open sets.
In particular, $\phi(D(c_0, \epsilon))$ is open.
Consider the boundary of this set, namely $\partial(\phi(D(c_0, \epsilon))$.
We first note that $\partial(\phi(D(c_0, \epsilon)) \subseteq \phi(C)$.
To see this, assume to the contrary that there exists a $z \in \partial(D(c_0, \epsilon)) \setminus \phi(C)$.
Then there would exist an $\epsilon_2 > 0$ such that $D(z, \epsilon_2) \cap \phi(C) = \emptyset$.
We see that $\phi^{-1}(D(z, \epsilon_2)) \cap D(c_0, \epsilon)$ is open, 
    and yet $D(z,\epsilon_2) \cap \phi(C)$ is not open.  
A contradiction, hence $\partial(\phi(D(c_0, \epsilon)) \subseteq \phi(C)$.

We have $\phi(c_0) \in \mathcal{M}(\alpha)$ and hence $|\phi(c_0)| \leq R_\alpha$.
We further see that $\partial(\phi(D(c_0, \epsilon)) \subseteq \phi(C)$ and 
    for all $c \in C$ that $|\phi(c)| > R_\alpha$.
This gives us that $\mathcal{M}(\alpha) \subseteq D(0, R_\alpha) \subseteq \phi(D(c_0, \epsilon)$.
Taking pre-images, we see that $D(c_0,\epsilon) \cap \mathcal{M}(\alpha)$ contains an uncountable number of points, hence is not a singleton.
\end{proof}

\section{Totally real algebraic numbers in $\mathcal{M}(\alpha)$}
\label{sec:Main}

In this section we prove 
Theorems \ref{thm:alpha} and \ref{thm:alpha = 1}. Let $\alpha\in\mathbb{C}$. Recall, we denote by
$$\mathrm{Prep}(\alpha)=\{c \in\mathbb{C} : f_c^n(\alpha)=f^m_c(\alpha),\,\,\text{for}\,\,m,n\in\mathbb{Z}_{\geq0}\}$$ the collection of parameters $c$ for which $\alpha$ is $f_c$-preperiodic. It is clear from the definition that $\mathrm{Prep}(\alpha)\subset \mathcal{M}(\alpha)$. It was shown by Baker-DeMarco \cite[Lemma 3.5]{BD11} that $\mathrm{Prep}(\alpha)$ is an infinite set. Also, $\mathrm{Prep}(0)$ is the set of complex parameters $c$ such that $f_c$ is post-critically finite (PCF) because the only finite critical point $0$ of $f_c$ has finite forward orbit.

The following lemma can be viewed as a local-to-global principle in the parameter space setting. A dynamical version was proved in \cite[Corollary 6.3]{CS97}. It is a useful auxiliary result to understand how to construct an adelic set in the proof of Theorem \ref{thm:alpha}.

\begin{lemma} \label{lem:prepinmandel} Suppose that $\alpha,c\in \overline{K}$. Then $\alpha$ is $f_c$-preperiodic if and only if $c$ belongs to the $v$-adic Mandelbrot set $\mathcal{M}_v(\alpha)$ at all places $v\in M_K$.
\end{lemma}

\begin{rem}
    In the case $\alpha=0$, it is equivalent to say that $f_c$ is post-critically finite if and only if $0$ is contained in the $v$-adic filled Julia set $\mathcal{K}_v(f_c)$  at all places $v\in M_K$ (i.e., $0$ stays bounded $v$-adically).
\end{rem}

\begin{proof}[Proof of Lemma \ref{lem:prepinmandel}] Using dynamical local-to-global principle by Call-Goldstine \cite[Corollary 6.3]{CS97}, we have
\begin{align*} \alpha\,\,\text{is preperiodic for}\,\,f_c &\Longleftrightarrow \alpha \in \mathcal{K}_v(f_c)\,\,\text{for all places}\,\,v\in M_K\\ 
&(\text{meaning that $\alpha$ stays bounded $v$-adically})\\&\Longleftrightarrow c\in\mathcal{M}_v(\alpha)\,\,\text{for all places}\,\,v\in M_K
\end{align*}
The last equivalent statement follows from the definition of the $v$-adic Mandelbrot set.
\end{proof}

It was recently proved by Noytaptim and Petsche \cite[Theorem 1]{NP23} that there are only finitely many totally real parameters $c$ for which $f_c$ is  PCF polynomials. More precisely, they showed that $\mathrm{Prep}(0)\cap\mathbb{Q}^{\mathrm{tr}}=\{-2,-1,0\}$. That is, $x^2, x^2-1$, and $x^2-2$ are the only   post-critically finite (quadratic) unicritical polynomials defined over the maximal totally real extension $\mathbb{Q}^{\mathrm{tr}}$.

 The following result provides a classification of  quadratic polynomials $f_c(x)$ defined over $\mathbb{Q}^{\mathrm{tr}}$ for which a rational number $\alpha$ is $f_c$-preperiodic. It turns out that there are only finitely many such polynomials $f_c(x)\in\mathbb{Q}^{\mathrm{tr}}[x]$ when $\alpha\in\mathbb{Q}\cap(-2,2)$. We have proved that the set $\mathcal{M}(\alpha)$ is contained in real line when a real number $\alpha$ satisfies $|\alpha|\geq 2$ (cf. Theorem \ref{thm:M is real}). Thus there are infinitely many quadratic polynomials defined over $\mathbb{Q}^{\mathrm{tr}}$ in which a rational number $\alpha$ with $|\alpha|\geq 2$ is $f_c$-preperiodic. 
 
\begin{proof}[Proof of Theorem \ref{thm:alpha}] 
First consider the case when $|\alpha| < 2$
Suppose that $\alpha$ is  a rational number such that $-2<\alpha<2$. Consider the compact adelic set $\mathbb{E}=\prod_{v\in M_{\mathbb{Q}}} E_v$ where 
$$E_v=\begin{cases} \mathcal{M}(\alpha)\cap \mathbb{R},&\mbox{if}\,\,v=\infty\\ \mathcal{M}_{v}(\alpha),&\mbox{if}\,\, v\neq\infty\end{cases}$$
The global capacity of $\mathbb{E}$ is 
$$\gamma_{\infty}(\mathbb{E})=\gamma_{\infty}(\mathcal{M}(\alpha)\cap\mathbb{R})\prod_{v\neq \infty}\gamma_{\infty}(\mathcal{M}_v(\alpha))<1$$ where the last inequality follows from the fact that $\gamma_{\infty}(\mathcal{M}_v(\alpha))=1$ at all places $v\neq \infty$ by Proposition \ref{prop: capmandel} and $\gamma_{\infty}(\mathcal{M}(\alpha)\cap\mathbb{R})<\gamma_{\infty}(\mathcal{M}(\alpha))=1$ by monotonicity of the logarithmic capacity \cite[Theorem 5.1.2(a)]{Ra95}. Hence by Fekete-Szeg\"{o} theorem \ref{thm:FeketeSzego}, we have $E_v$ contains finitely many complete conjugate sets algebraic numbers for each $v\in M_{\mathbb{Q}}$. Hence the finiteness follows.

It remains to check that the set  $\mathrm{Prep}(\alpha)\cap\mathbb{Q}^{\mathrm{tr}}$ is always nonempty when it is finite for $|\alpha|<2.$ It is straightforward to verify that $\theta_{\alpha}=-\alpha^2-1-\sqrt{\alpha^2+1}\in\overline{\mathbb{Q}}$ is contained in $\mathrm{Prep}(\alpha)\cap\mathbb{Q}^{\mathrm{tr}}$ for every rational number $\alpha \in (-2,2)$. In fact, $\theta_{\alpha}$ is a totally real algebraic number because it satisfies the minimal polynomial $G_{\theta_{\alpha}}(X)=X^2+2(\alpha^2+1)X+\alpha^4+\alpha^2=0$ and its Galois conjugate (if any\footnote{The reason we say ``if any" because it could be possible that some rational numbers $\alpha\in (-2,2)$ the algebraic number $\theta_{\alpha}$ is simply a rational integer. For example, $\theta_0=-2.$}) $\theta'_{\alpha}=-\alpha^2-1+\sqrt{\alpha^2+1}$ is also algebraic number contained in the same closed real interval $\mathcal{M}(\alpha)\cap\mathbb{R}.$ To see that $\theta_{\alpha}$ is in $\mathrm{Prep}(\alpha)\cap\mathbb{Q}^{\mathrm{tr}}$, it suffices to verify that the forward orbit of $\alpha$ under $f_{\theta_{\alpha}}$-iteration is finite. Indeed, $f^2_{\theta_{\alpha}}(\alpha)=f^3_{\theta_{\alpha}}(\alpha).$ Hence $\#\mathrm{Prep}(\alpha)\cap\mathbb{Q}^{\mathrm{tr}}\geq 1$ as desired.
This proves part \eqref{case:alpha<2}.

Next suppose that $\alpha$ is a rational number such that $\alpha\in (-\infty, -2]\cup [2,\infty).$  

Without invoking a powerful Fekete-Szeg\"{o} theorem, we can check directly from Theorem \ref{thm:M is real} that $\mathrm{Prep}(\alpha)\cap\mathbb{Q}^{\mathrm{tr}}$ is an infinite set. Indeed, if $c\in\mathrm{Prep}(\alpha)$, then so are all of its conjugates and $$\mathrm{Prep}(\alpha)\subseteq \mathcal{M}(\alpha)\subseteq\mathbb{R}.$$ As $\mathrm{Prep}(\alpha)$ is an infinite set, and so it follows that $\mathrm{Prep}(\alpha)\cap\mathbb{Q}^{\mathrm{tr}}$ must be also be infinite.  
This proves part \eqref{case:alpha>=2}.
\end{proof}

\begin{rem} \label{rem:equipfoffiniteness}
 Alternatively, one can also appeal to the arithmetic equidistribution of small 
height to deduce the finiteness of the intersection $\#(\mathrm{Prep}(\alpha)\cap\mathbb{Q}^{\mathrm{tr}})$. As the proof contains different ingredients from the previous approach, we include details here for interested reader. Let us recall some basic set-up of height function associated to the adelic set. Let $K$ be a number field. Define a compact adelic set by  $\mathbb{M}_{\alpha}=\prod_{v\in M_K}\mathcal{M}_{v}(\alpha)$ with $\mathcal{M}_v(\alpha)=\mathcal{D}(0,1)$ for all but finitely many $v\in M_K$. Thus the global capacity of $\mathbb{M}_{\alpha}$ is $\gamma_{\infty}(\mathbb{M}_{\alpha})=1$. Baker and DeMarco \cite[\S3.3]{BD11} introduced the height function associated to $\mathbb{M}_{\alpha}$ and defined $h_{\mathbb{M}_{\alpha}}: \overline{K}\rightarrow[0,+\infty)$ by $$h_{\mathbb{M}_{\alpha}}(x)=\frac{1}{|S|}\sum_{v\in M_{\mathbb{Q}}}N_v\sum_{x\in S}G_{\mathcal{M}_{v}(\alpha)}(x)$$ where $S$ denotes a finite set of Galois conjugates $\mathrm{Gal}(\overline{K}/K)\cdot x$ of $x\in \overline{K}$ and $G_{\mathcal{M}_v(\alpha)} : \mathbb{C}_v\rightarrow [0,+\infty)$ is a continuous local Green's function  which detects points in $\mathcal{M}_v(\alpha)$ in the sense that $G_{\mathcal{M}_v(\alpha)}(x)=0$ if and only if $x\in \mathcal{M}_v(\alpha).$

Suppose, for the sake of contradiction, that there exists an infinite distinct sequence  $\{c_n\}\subset \overline{\mathbb{Q}}\cap\mathbb{Q}^{\mathrm{tr}}$ such that $\alpha$ is $f_{c_n}$-preperiodic. This means that the sequence $\{c_n\}$ belongs to $\mathcal{M}_v(\alpha)$ for all $v\in M_{\mathbb{Q}}$. Thus $h_{\mathbb{M}_{\alpha}}(c_n)\rightarrow0$  as $n\rightarrow\infty$ (i.e., $\{c_n\}$ is small points). Together with the fact that the global capacity $\gamma_{\infty}(\mathbb{M}_{\alpha})=1$, we then apply the arithmetic equidistribution theorem (cf. \cite[Corollary 7.53]{BR10}) to obtain that the sequence $\{c_n\}$ is equidistributed with respect to the equilibrium measure $\mathbb{\mu}_{\mathcal{M}_v(\alpha)}$ on $\mathbb{P}^1_{\mathrm{Berk},\mathbb{C}_v}$ (i.e., the Berkovich projective line, see \cite{Be19} and \cite{BR10} for exposition) at all places $v\in M_{\mathbb{Q}}$. At $v=\infty$, the measure  $\mu_{\mathcal{M}(\alpha)}$ is supported on a subset contained in real line. On the other hand, the support of $\mu_{\mathcal{M}(\alpha)}$ on $\mathbb{P}^1(\mathbb{C})$  is the boundary of the Mandelbrot set $\partial\mathcal{M}(\alpha)$ \cite[Proposition 3.3(5)]{BD11} which is not contained in real line as $\alpha\in \mathbb{Q}\cap(-2,2)$ and it gives rise to a contradiction. Hence the intersection $\mathrm{Prep}(\alpha)\cap\mathbb{Q}^{\mathrm{tr}}$ is finite as desired.
\end{rem}

For special\footnote{The rational numbers $\alpha=\pm1$ are more ``special" than other values in $(-2,2)$ in the following two senses : first, each totally real parameter $c$ in $\mathrm{Prep}(\alpha)\cap\mathbb{Q}^{\mathrm{tr}}$ is algebraic integer and; second, the real part of $\mathcal{M}(\alpha)\cap\mathbb{R}$ has logarithmic capacity less than $1$.} values of rational numbers $\alpha=\pm1$, we are able to explicitly compute the intersection $\mathrm{Prep}(\alpha)\cap\mathbb{Q}^{\mathrm{tr}}$.

\begin{rem} Before proceeding the proof, we make a brief discuss about the statement of the result.  The numerical criterion (Theorem \ref{thm : numerical tool}) allows us to bound the degree of \textbf{algebraic integers} in shorts intervals. As for other rational numbers $\alpha \in (-2,2)\cap\mathbb{Q}$ that differ from $\pm1$, we could not apply the numerical criterion to bound the degree of $c\in\mathrm{Prep}(\alpha)\cap\mathbb{Q}^{\mathrm{tr}}$ because $c$ is no longer algebraic integer. 
\end{rem}

\begin{proof}[Proof of Theorem \ref{thm:alpha = 1}]
We wish to obtain explicit computation of the intersection $\mathrm{Prep}(1)\cap\mathbb{Q}^{\mathrm{tr}}$ and this immediately yields  the same result of $\mathrm{Prep}(-1)\cap\mathbb{Q}^{\mathrm{tr}}$ as $f_c$ is even (meaning that $f_c(x)=f_c(-x)$). We are seeking totally real algebraic parameters $c$ for which $1$ is $f_c$-preperiodic. That is, parameters $c\in\mathbb{Q}^{\mathrm{tr}}$ satisfy the polynomial relation $F(X)=f^i_X(1)-f^j_X(1)$ for some integers $0\leq i<j$. It is clear that $F(X)$ is a monic polynomial with integer coefficients and so $c$ must be algebraic integer. In order to search for totally real algebraic integers in the real part $\mathcal{M}(1)\cap\mathbb{R}$, it suffices to search for algebraic integers in the closed interval $[-2-\sqrt{2},0].$ Consider the sequences, for any $n\geq 2$,
$$a_n=d_n([-2-\sqrt{2},0])^{n(n-1)}=(2+\sqrt{2})^{n(n-1)}D_n\quad\text{and}\quad b_n=\frac{n^{2n}}{n!^2}$$ where   $ [-2-\sqrt{2},0]$  is the smallest interval (of length less that $4$)  that contained the real part $\mathcal{M}(1)\cap \mathbb{R}$ (cf. Theorem \ref{thm:M is bounded}) and the sequence $\{D_n\}_{n\geq 2}$ given in Theorem \ref{thm : numerical tool}. Applying Theorem \ref{thm : numerical tool} with the sequences $\{a_n\}$ and $\{b_n\}$, we found $n_0=12$ such that $$a_{12}<b_{12}\quad\text{and}\quad a_{13}/a_{12}<b_{13}/b_{12}.$$ This means that the interval  $ [-2-\sqrt{2},0]$ contains finite complete conjugate sets of algebraic integers of degree at most $11$. 
Consider a polynomial $p(c)$ with all roots in $[-2 - \sqrt{2}, 0]$.
We see that $z^{\deg(p)} p(z+1/z+2)$ is a monic integer polynomial with all roots contained on or inside the unit disc.
By Kronecker's Theorem \cite{Kro57} this implies that $z^{\deg(p)} p(z+1/z+2)$ is a cyclotomic polynomial of degree at most 22.
There are only a finite number of such polynomials, which are easily checked.
We verify that 
$c, c+1, c+2, c+3,$ and $c^2+4c+2$ are monic polynomials with integers coefficients whose roots are contained in $[-2-\sqrt{2},0]$. Then  $\mathrm{Prep}(\alpha)\cap\mathbb{Q}^{\mathrm{tr}}\subseteq\{-2-\sqrt{2},-3,-2,-1,-2+\sqrt{2},0\}:=A$ and it is straightforward to check that each element $c\in\overline{\mathbb{Z}}$ in the set $A$ satisfy the property that $1$ is $f_c$-preperiodic. Hence we have $$\mathrm{Prep}(1)\cap\mathbb{Q}^{\mathrm{tr}}=\{-2-\sqrt{2},-3,-2,-1,-2+\sqrt{2},0\}.$$
\end{proof}

\section{Open Questions}
\label{sec:conc}

 We close our article by posing some questions for further investigation.
\begin{enumerate}
    \item It would be interesting to classify for which $\alpha\in\overline{\mathbb{Q}}\setminus \mathbb{Q}$ such that the set $\{c\in\mathbb{Q}^{\mathrm{tr}} : f^n_c(\alpha)=f^m_c(\alpha),\,\,0\leq m<n\}$ is finite.
    It is not clear what the correct generalization should be when $\alpha \not\in \mathbb{Q}$.
    For example, if $\alpha = \sqrt{2}-1 \in \mathbb{Q}^{\mathrm{tr}}$, then it is easy
    to check that $c= -2+\sqrt{2} \in \mathrm{Prep}(\alpha)$ as $\{f_c^n(\alpha)\}_{n=0}^\infty = 
    \{-2+\sqrt{2}, 1+\sqrt{2}, 1+\sqrt{2}, 1+\sqrt{2}, \dots\}$.
    One can also quickly check that the Galois conjugate of $c$ over $\mathbb{Q}$ is not in 
    $\mathrm{Prep}(\alpha)$.
    In general if $\alpha$ is a preperiodic point for $c$, then $\alpha$ is a preperiodic point
    for the Galois conjugates of $c$ {\bf over $\mathbb{Q}(\alpha)$}.
    There will be cases where $c$ is totally real when respect to the field $\mathbb{Q}(\alpha)$, whereas it is not totally real with respect to $\mathbb{Q}$.
    \item In the same vein to Theorem \ref{thm:alpha}, it would be interesting to extend the result to the unicritical family $f_c(x)=x^d+c$ when degree $d\geq 3$.
    \item The arithmetic equidistribution theorem used to derive the finiteness of $\mathrm{Prep}(\alpha)\cap\mathbb{Q}^{\mathrm{tr}}$ (cf. Remark \ref{rem:equipfoffiniteness}) does not provide an effective result as we have no control over the number of parameters $c\in\mathbb{Q}^{\mathrm{tr}}$ such that $\alpha\in \mathbb{Q}$ is $f_c$-preperiodic. It is reasonable to give an effective computable bounds or uniform bounds for the size of $ \mathrm{Prep}(\alpha)\cap\mathbb{Q}^{\mathrm{tr}}$ when $\alpha\in (-2,2)\cap\mathbb{Q}$ in light of DeMarco-Krieger-Ye \cite{DKY22} and Fili \cite{Fi17} building upon the quantitative equidistribution theorem of Favre-Rivera-Letelier \cite{FRL06}. Also, it would  be interesting to develop a different approach to compute such intersection for more values of rational numbers in $(-2,2)$.
\end{enumerate}

\end{document}